\documentclass[11pt]{article}
\usepackage{amssymb,amsmath,amsfonts,amsthm,authblk,hyperref,here,enumerate,xfrac,mathtools,graphicx,pstool,xcolor,pgf,tikz}
\usetikzlibrary{arrows,arrows.meta}
\usetikzlibrary{decorations.pathreplacing,decorations.markings,hobby}

\usepackage[OT2,OT1]{fontenc} 
\newcommand\cyr
{
\renewcommand\rmdefault{wncyr} \renewcommand\sfdefault{wncyss} \renewcommand\encodingdefault{OT2} \normalfont
\selectfont
}

\DeclareTextFontCommand{\textcyr}{\cyr} \def\cprime{\char"7E }

\numberwithin{equation}{section}

\newcommand{\Z}{\mathbb Z}
\newcommand{\C}{\mathbb C}
\definecolor{darkgreen}{rgb}{0,.6,0}

\definecolor{MyDarkBlue}{rgb}{0,0.29,0.7}
\hypersetup{
    colorlinks=true,       
    linkcolor=MyDarkBlue,          
    linkbordercolor=MyDarkBlue,
    citecolor=MyDarkBlue,
    citebordercolor=MyDarkBlue,        
    filecolor=MyDarkBlue,      
    urlcolor=MyDarkBlue           
}

\theoremstyle{plain}

\newtheorem{theorem}{Theorem}[section]
\newtheorem*{theorem*}{Theorem}

\newtheorem{lemma}[theorem]{Lemma}
\newtheorem{prop}[theorem]{Proposition}

\theoremstyle{definition}

\newtheorem{definition}[theorem]{Definition}
\newtheorem{example}[theorem]{Example}

\tikzset{
  on each segment/.style={
    decorate,
    decoration={
      show path construction,
      moveto code={},
      lineto code={
        \path [#1]
        (\tikzinputsegmentfirst) -- (\tikzinputsegmentlast);
      },
      curveto code={
        \path [#1] (\tikzinputsegmentfirst)
        .. controls
        (\tikzinputsegmentsupporta) and (\tikzinputsegmentsupportb)
        ..
        (\tikzinputsegmentlast);
      },
      closepath code={
        \path [#1]
        (\tikzinputsegmentfirst) -- (\tikzinputsegmentlast);
      },
    },
  },
  mid arrow/.style={postaction={decorate,decoration={
        markings,
        mark=at position .5 with {\arrow[#1]{>}}
      }}},
}

\begin{document}

\title{Non-integer valued winding numbers and a generalized Residue Theorem}
\author[1]{Norbert Hungerb\"uhler}
\author[2]{Micha Wasem}
\affil[1]{Department of Mathematics, ETH Z\"urich, R\"amistrasse 101, 8092 Z\"urich, Switzerland}
\affil[2]{HTA/HSW Freiburg, HES-SO University of Applied Sciences and Arts Western Switzerland, P\'erolles 80/Chemin du Mus\'ee 4, 1700 Freiburg, Switzerland}
\date{\today}
\maketitle
\begin{abstract}
\noindent We define a generalization of the winding number of a piecewise $C^1$ cycle in the complex plane
which has a geometric meaning also for points which lie {\em on\/} the cycle. 
The computation of this winding number relies on the Cauchy principal
value, but is also possible in a real version via an integral with
bounded integrand.
The new winding number allows to establish a generalized
residue theorem which covers also the situation where
singularities lie on the cycle. This residue theorem can be used 
to calculate the value of  improper integrals for which the standard
technique with the classical residue theorem does not
apply. 
\end{abstract}
\hspace{5ex}{\small{\it Key words\/}: Cauchy principal value, winding number, residue theorem}

\hspace{5ex}{\small{\it 2010 Mathematics Subject 
Classification\/}: {\bf 30E20}
\normalsize
\section{Introduction}
One of the most prominent tools in complex analysis is Cauchy's Residue Theorem.
To state the classical version of this theorem (see, e.g.,~\cite{ahlfors} or \cite[Theorem 1, p. 75]{narasimhan})
we briefly recall the following notions: A {\em chain\/} is a finite formal
linear combination
$$
\Gamma=\sum_{\ell=1}^km_\ell \gamma_\ell,\quad m_\ell\in\mathbb Z,
$$
of continuous curves $\gamma_\ell:[0,1]\to\mathbb C$. 
A {\em cycle\/} $C$ is a chain, where every point $a\in\mathbb C$ is, 
counted with the corresponding multiplicity $m_\ell$,
as often a starting point of a curve $\gamma_\ell$ as it is an endpoint.
A cycle $C$ is {\em null-homologous\/} in $D\subset\mathbb C$, if its 
winding number for all points in $\mathbb C\setminus D$ vanishes.
Equivalently, $C$ is null-homologous in $D$, if it can be written as 
a linear combination of closed curves wich are contractible in $D$.
Then the residue theorem can be expressed as follows:
\begin{theorem*}[Classical Residue Theorem] Let $U\subset\mathbb C$ be an open set
and $S\subset U$ be a set without  accumulation points in $U$ such that
$f:U\setminus S\to \mathbb C$ is holomorphic. Furthermore, let 
$C$ be a null-homologous  cycle in $U\setminus S$. Then there holds
\begin{equation}\label{classical-residue-theorem}
\frac1{2\pi \mathrm i}\oint_C f(z)\,\mathrm dz = \sum_{s\in S}n_s(C)\operatorname{res}_s f(z),
\end{equation}
where $n_s(C)$ denotes the winding number of $C$ with respect to $s$.
\end{theorem*}
Henrici considers in~\cite[Theorem 4.8f]{henrici} a version of the residue theorem where $C$ is
the boundary of a semicircle in the upper half-plane with diameter $[-R,R]$, and
where $f$ is allowed to have poles on $(-R,R)$ which involve odd powers only. The result is basically 
a version of the classical formula~(\ref{classical-residue-theorem}), but with
winding number $\frac12$ for the singularities on the real axis, and
where the integral on the left hand side of~(\ref{classical-residue-theorem}) is
interpreted as a Cauchy Principal Value.
Another very recent version of the residue theorem, where poles of order 1 on the piecewise $C^1$ boundary curve $\gamma$ of
an open set are allowed, is discussed in~\cite[Theorem 1]{legua}. There, if
a pole is sitting on a corner of $\gamma$, the winding number is replaced by
the angle formed by $\gamma$ in this point, divided by $2\pi$.
In~\cite{lu} a version of the residue theorem for functions $f$ with finitely many poles is presented where singularities in points $z_0$ on a closed curve $\Gamma$ are allowed provided $|f(z)|=O(|z-z_0|^\mu)$ near $z_0$, with $-1<\mu\le 0$.
Further versions of generalizations of the residue theorem are described 
in~\cite{zhong2} and~\cite{zhong1}: there, versions for unbounded multiply connected regions of the second class are applied to higher-order singular integrals
and transcendental singular integrals.

In the present article, we introduce a generalized, non-integer winding number 
for  piecewise $C^1$ cycles $C$, and a general version of the residue theorem
which covers all cases of singularities on $C$. We will assume throughout the article that all curves are continuous. In particular, a piecewise $C^1$ curve is a continuous curve which is piecewise $C^1$. Recall that a closed piecewise $C^1$ immersion $\Lambda:[a,b]\to \C$ is a closed curve such that there is a partition $a=a_0<a_1<\ldots<a_n=b$ such that $\Lambda|_{[a_k,a_{k+1}]}$ is of class $C^1$ and such that
$\dot\Lambda|_{[a_k,a_{k+1}]}\ne 0$ for all $k=0,\ldots, n-1$. If $\dot\Lambda|_{[a_k,a_{k+1}]}$ is furthermore a Lipschitz function for all $k=0,\ldots, n-1$, then $\Lambda$ is called a closed piecewise $C^{1,1}$ immersion.

\section{A Generalized Winding Number}\label{generalizedwindingnumber}
The aim of this section is to generalize the winding number
to  piecewise $C^1$ cycles with respect to points sitting on the cycle itself.

The usual standard situation is the following: The winding number of a closed 
piecewise $C^1$ curve $\gamma:[a,b]\to \C\setminus\{0\}$ around $z=0$ is given by
$$
n_0(\gamma)=\frac{1}{2\pi\mathrm i}\oint_\gamma \frac{\mathrm dz}{z} \in\mathbb Z,
$$
see for example \cite[p. 70]{narasimhan} or \cite[p. 75]{baker}. More generally, one can 
replace the curve $\gamma$ by a  piecewise $C^1$ cycle $C
=\sum_{\ell=1}^k m_\ell\gamma_\ell$
. An integral over the cycle is then
$$\oint_C f(z)\,\mathrm dz:=\sum_{\ell=1}^km_\ell\oint_{\gamma_\ell} f(z)\,\mathrm dz.$$

In order to make sense of the winding number also for points {\em on\/} the curve, we 
use the Cauchy principal value:
\begin{definition}\label{def-winding}
The winding number of a  
piecewise $C^1$ cycle $C:[a,b]\to \C$ with respect to $z_0\in C$ is 
\begin{equation}\label{windungszahl}
n_{z_0}(C):=\operatorname{PV}\frac{1}{2\pi\mathrm i}\oint_C \frac{\mathrm dz}{z-z_0}=
\lim_{\varepsilon\searrow 0}\frac{1}{2\pi\mathrm i}
\int\limits_{|C-z_0|>\varepsilon} \frac{\mathrm dz}{z-z_0}\,.
\end{equation}
\end{definition}
It is not a priori clear whether this limit exists and what its geometric meaning is.
So, we start by looking at the following model case: Using the Cauchy principal value we can easily compute the winding number with respect to $z=0$ 
of the {\em model sector-curve}  $\gamma = \gamma_1 + \gamma_2 + \gamma_3$, where
$$\begin{aligned}
\gamma_1: [0,r]\to\C, &~~t\mapsto t,\ \ \gamma_1'(t) = 1\\
\gamma_2:[0,\alpha]\to\C, &~~t\mapsto r\mathrm e^{\mathrm it},\ \ \gamma_2'(t) = r\mathrm i \mathrm e^{\mathrm it}\\
\gamma_3:[0,r]\to\C, &~~t\mapsto (r-t)\mathrm e^{\mathrm i\alpha},\ \ \gamma_3'(t) = -\mathrm e^{\mathrm i\alpha}
\end{aligned}$$
for $\alpha\in[0,2\pi]$ (see Figure~\ref{fig-gamma}).
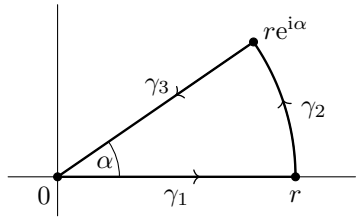
\begin{figure}[h]
\begin{center}
\begin{tikzpicture}[line cap=round,line join=round,x=18,y=18]
\draw[color=black] (-1.04,0.) -- (6.26,0.);
\draw[color=black] (0.,-0.82) -- (0.,3.62);
\clip(-1.04,-0.82) rectangle (6.26,3.62);
\draw [shift={(0.,0.)},line width=0.4pt] (0,0) -- (0.:1.3) arc (0.:34.55811897572532:1.3) -- cycle;
\draw [line width=.9pt] (0.,0.)-- (5.,0.);

\draw [->,line width=.5pt] (0.,0.)-- (3,0.);

\draw [shift={(0.,0.)},line width=.9pt]  plot[domain=0.:0.6031529594223372,variable=\t]({1.*5.*cos(\t r)+0.*5.*sin(\t r)},{0.*5.*cos(\t r)+1.*5.*sin(\t r)});

\draw [->,shift={(0.,0.)},line width=.5pt]  plot[domain=0.:0.6031529594223372*.55,variable=\t]({1.*5.*cos(\t r)+0.*5.*sin(\t r)},{0.*5.*cos(\t r)+1.*5.*sin(\t r)});

\draw [line width=.9pt] (4.117756103310844,2.836209560953897)-- (0.,0.);
\draw [->,line width=.5pt] (4.117756103310844,2.836209560953897)-- (3*4.117756103310844/5,3*2.836209560953897/5);

\begin{small}
\draw [fill=black] (0.,0.) circle (1.5pt);
\draw[color=black] (-0.28,-0.4) node {$0$};
\draw [fill=black] (5.,0.) circle (1.5pt);
\draw[color=black] (5.,-0.4) node {$r$};
\draw[color=black] (2.5,-0.44) node {$\gamma_1$};
\draw [fill=black] (4.117756103310844,2.836209560953897) circle (1.5pt);
\draw[color=black] (4.8,3.13) node {$r\mathrm e^{\mathrm i\alpha}$};
\draw[color=black] (5.39,1.4) node {$\gamma_2$};
\draw[color=black] (2.1,1.9) node {$\gamma_3$};
\draw[color=black] (1.,0.3) node {$\alpha$};
\end{small}
\end{tikzpicture}
\caption{The model sector-curve $\gamma=\gamma_1+\gamma_2+\gamma_3$.}\label{fig-gamma}
\end{center}
\end{figure}
Since
$$\begin{aligned}
\operatorname{PV}\oint_\gamma \frac{\mathrm dz}{z} & = \lim_{\varepsilon\searrow 0}\left(\int_\varepsilon^r \frac{\mathrm dt}{t}+\int_0^\alpha\frac{r\mathrm i \mathrm e^{\mathrm it}}{r \mathrm e^{\mathrm it}}\,\mathrm dt + \int_0^{r-\varepsilon}\frac{-\mathrm e^{\mathrm i\alpha}}{(r-t)\mathrm e^{\mathrm i\alpha}}\,\mathrm dt\right)=\\
& = \lim_{\varepsilon\searrow 0}(\ln r-\ln \varepsilon+\mathrm i\alpha+\ln\varepsilon - \ln r) = \mathrm i\alpha,
\end{aligned}$$
we obtain
$$
n_0(\gamma)=\operatorname{PV}\frac{1}{2\pi\mathrm i}\oint_\gamma \frac{\mathrm dz}{z}=\frac{\alpha}{2\pi}
$$
with a meaningful geometrical interpretation.
  
Consider now a closed piecewise $C^1$ immersion $\Gamma:[a,b]\to\C$ starting and ending in $0$ but
such that $\Gamma(t)\ne 0$ for all $t\ne a,b$ 
and such that the (positively oriented) angle between $\lim_{t\searrow a}\dot\Gamma(t)$ and $-\lim_{t\nearrow b}\dot\Gamma(t)$ equals $\alpha\in[0,2\pi]$. 
By a suitable rotation we may assume, without loss of generality, that  $\lim_{t\searrow a}\dot\Gamma(t)$ is a positive real number (see Figure~\ref{fig-re}).
\begin{figure}[h]
\begin{center}
\begin{tikzpicture}[x=16,y=16]
\clip(-.8,-0.65) rectangle (7,4.5);
\draw (-.8,0) -- (7,0);
\draw (0,-.6) -- (0,5);
\draw[line width=.9pt,smooth,domain=-0.523599:0.523599] plot ({6*cos(3*\x r)*(0.777714*cos(\x r) - 0.880066*sin(\x r))},{6*cos(3*\x r)*(0.628619*cos(\x r) + 1.0888*sin(\x r))});
\draw[->,line width=.5pt,smooth,domain=-.01:.01] plot ({6*cos(3*\x r)*(0.777714*cos(\x r) - 0.880066*sin(\x r))},{6*cos(3*\x r)*(0.628619*cos(\x r) + 1.0888*sin(\x r))});
\draw[color=black] (5.4,3) node {\small$\Gamma$};

\draw[color=black] (.5,0.4) node {\small$\alpha$};
\draw [fill=black] (0.,0.) circle (1.5pt);
\draw[color=black] (-0.28,-0.4) node {$0$};

\draw [line width=0.4pt] (0.997564, 0.0697565) arc (4:76:1) ;


\end{tikzpicture}
\caption{Winding number with respect to the origin sitting on the curve $\Gamma$.}\label{fig-re}
\end{center}
\end{figure}
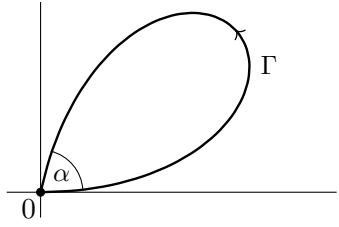
We assume that $\Gamma$ is homotopic to a model sector-curve $\gamma$ with the same angle $\alpha$
in the following sense: There is a continuous function $H:[a,b]\times [0,1]\to \mathbb C$ such that
\begin{equation}\label{eq-hom} 
\arraycolsep2pt
\begin{array}{rll} 
H(t,0) & = \Gamma(t) &\text{for all $t\in[a,b]$,}\\
H(t,1) & = \gamma(t) &\text{for all $t\in[a,b]$,}\\
0=H(a,s) &= H(b,s)\quad   &\text{for all $s\in[0,1]$,}\\
H(t,s)& \neq 0 &\text{for all $t\in(a,b)$ and all $s\in[0,1]$}.
\end{array}
\end{equation}
Then we claim that
\begin{equation}\label{eq-claim}
\lim_{\varepsilon\searrow 0}\int_{|\Gamma|>\varepsilon}\frac{\mathrm dz}{z}=\lim_{\varepsilon\searrow 0}\int_{|\gamma|>\varepsilon}\frac{\mathrm dz}{z}\,.
\end{equation}

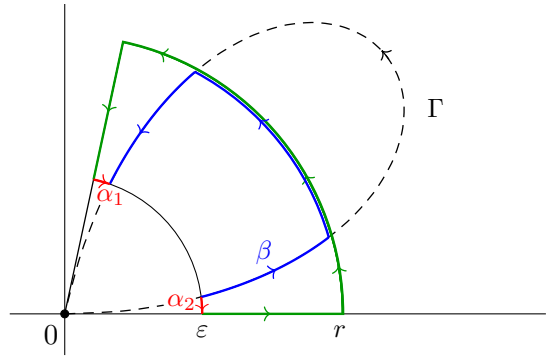
\begin{figure}[h]
\begin{center}
\begin{tikzpicture}[x=26,y=26]
\clip(-.8,-0.6) rectangle (7,4.5);

\draw (-.8,0) -- (7,0)  node[anchor=west] {};
\draw (0,-.6) -- (0,5)  node[anchor=east] {};
\draw[line width=.5pt,dash pattern=on 4pt off 3pt,smooth,domain=-0.523599:0.523599] plot ({6*cos(3*\x r)*(0.777714*cos(\x r) - 0.880066*sin(\x r))},{6*cos(3*\x r)*(0.628619*cos(\x r) + 1.0888*sin(\x r))});
\draw[->,line width=.5pt,smooth,domain=-.01:.01] plot ({6*cos(3*\x r)*(0.777714*cos(\x r) - 0.880066*sin(\x r))},{6*cos(3*\x r)*(0.628619*cos(\x r) + 1.0888*sin(\x r))});
\draw[color=black] (5.4,3) node {\small$\Gamma$};
\draw[color=black] (2,-.26) node {\small$\varepsilon$};
\draw[color=black] (4,-.26) node {\small$r$};
\draw[color=blue] (2.9,.9) node {\small$\beta$};

\draw [fill=black] (0.,0.) circle (1.5pt);
\draw[color=black] (-0.2,-0.3) node {$0$};

\draw [color=black,shift={(0.,0.)},line width=0.5pt] (0.:4.05) arc (0:77.8966:4.05) to (0,0);

\draw [color=darkgreen,shift={(0.,0.)},line width=0.9pt] (2,0) to (0.:4.05) arc (0:77.8966:4.05) to (0.419355, 1.95554);
\draw [->,color=darkgreen,shift={(0.,0.)},line width=0.5pt] (2,0) to (0.:3) ;
\draw [->,color=darkgreen,shift={(0.,0.)},line width=0.5pt] (2,0) to (0.:4.04) arc (0:10:4.04);
\draw [->,color=darkgreen,shift={(0.,0.)},line width=0.5pt] (2,0) to (0.:4.04) arc (0:30:4.04);
\draw [->,color=darkgreen,shift={(0.,0.)},line width=0.5pt] (2,0) to (0.:4.04) arc (0:70:4.04);
\draw [->,color=darkgreen,shift={(0.,0.)},line width=0.5pt] (2,0) to (0.:4.04) arc (0:77.8966:4.04) to (0.629032, 2.93331);

\draw [shift={(0.,0.)},line width=0.4pt] (0.:2) arc (0:77.8966:2) ;
\draw [line width=0.9pt,color=blue] (3.84168, 1.11423) arc (16.1741:61.7224:4) ;
\draw [->,line width=0.5pt,color=blue] (3.84168, 1.11423) arc (16.1741:45:4) ;

\draw[color=blue,line width=.9pt,smooth,domain=-0.418649:-0.291346] plot ({6*cos(3*\x r)*(0.777714*cos(\x r) - 0.880066*sin(\x r))},{6*cos(3*\x r)*(0.628619*cos(\x r) + 1.0888*sin(\x r))});
\draw[color=blue,line width=.9pt,smooth,domain=0.291346:0.418649] plot ({6*cos(3*\x r)*(0.777714*cos(\x r) - 0.880066*sin(\x r))},{6*cos(3*\x r)*(0.628619*cos(\x r) + 1.0888*sin(\x r))});

\draw[->,color=blue,line width=.5pt,smooth,domain=-0.418649:-0.35] plot ({6*cos(3*\x r)*(0.777714*cos(\x r) - 0.880066*sin(\x r))},{6*cos(3*\x r)*(0.628619*cos(\x r) + 1.0888*sin(\x r))});
\draw[->,color=blue,line width=.5pt,smooth,domain=0.291346:.37] plot ({6*cos(3*\x r)*(0.777714*cos(\x r) - 0.880066*sin(\x r))},{6*cos(3*\x r)*(0.628619*cos(\x r) + 1.0888*sin(\x r))});

\draw [color=red,line width=0.9pt] (0:2) arc (0:7.02805:2) ;
\draw [-<,color=red,line width=0.4pt] (0:2) arc (0:4:2) ;

\draw [color=red,line width=0.9pt] (0.419355, 1.95554) arc (77.8966:70.8685:2) ;
\draw [->,color=red,line width=0.5pt] (0.419355, 1.95554) arc (77.8966:72:2) ;
\draw [color=white,fill=white ](.65,1.7) circle (.13);
\draw[color=red] (.67,1.7) node[anchor=center] {\small$\alpha_1$};

\draw [color=white,fill=white ](1.58,.17) circle (.14);
\draw [color=white,fill=white ](1.7,.22) circle (.1);
\draw [color=white,fill=white ](1.8,.22) circle (.1);

\draw[color=red] (1.7,.17) node[anchor=center]  {\small$\alpha_2$};


\end{tikzpicture}
\caption{The curves $\Gamma$ and $\gamma$.}\label{fig-cauchy}
\end{center}
\end{figure}

For $0<\varepsilon<r$ small enough we have (see Figure~\ref{fig-cauchy} for the definition of $\beta$)
$$
\int_{|\Gamma|>\varepsilon}\frac{\mathrm dz}{z}=\int_{\beta}\frac{\mathrm dz}{z}
$$
by Cauchy's integral theorem, and hence
$$\begin{aligned}
\int_{|\Gamma|>\varepsilon}\frac{\mathrm dz}{z}-\int_{|\gamma|>\varepsilon}\frac{\mathrm dz}{z} & = \int_{\beta}\frac{\mathrm dz}{z}-\int_{|\gamma|>\varepsilon}\frac{\mathrm dz}{z}=\\
& = \int_{\beta}\frac{\mathrm dz}{z}-\int_{|\gamma|>\varepsilon}\frac{\mathrm dz}{z}+\int_{(|\gamma|>\varepsilon) + \alpha_1-\beta+\alpha_2}\frac{\mathrm dz}{z}=\\
& = \int_{\alpha_1+\alpha_2}\frac{\mathrm dz}{z}\,.
\end{aligned}$$
Since
$$
\left|\int_{\alpha_1+\alpha_2}\frac{\mathrm dz}{z}\right| \leq \frac{1}{\varepsilon}\underbrace{\operatorname{Length}(\alpha_1+\alpha_2)}_{=o(\varepsilon)}
\longrightarrow 0 \text{ for $\varepsilon\searrow 0$},
$$
the claim~(\ref{eq-claim}) follows. Thus we get the geometrically reasonable result that
the winding number of the curve $\Gamma$ with respect to $z=0$ is, as we have just seen,
the angle $\alpha$ divided by $2\pi$:
\begin{equation}\label{eq-med}
n_0(\Gamma)=\operatorname{PV}\frac{1}{2\pi\mathrm i}\oint_\Gamma \frac{\mathrm dz}{z}=\frac{\alpha}{2\pi}\,.
\end{equation}
\begin{figure}[h]
\begin{center}
\begin{tikzpicture}[x=26,y=26,use Hobby shortcut,my marks/.style={},
    arrow mark/.style={
      /my marks/.append style={
        mark=at position #1 with {\arrow[line width=.5pt]{>}},
      },
    },
       do marks/.style={
      decorate,
      postaction={
        decoration={
          markings,
          /my marks,
        },
        decorate,
      },
    }]
\clip(-7.5,-4.1) rectangle (7.3,5.5);
\draw [color=black,line width=0.5pt] (0.321242, 0.383153) arc (47.2943:360-34.5:.5) ;

\draw (-8,0) -- (7.5,0)  node[anchor=west] {};
\draw (0,-4) -- (0,5.8)  node[anchor=east] {};
\draw[color=blue,->,line width=.5pt,smooth,domain=-.4:-.42] plot ({6*cos(3*\x r)*cos(\x r)},{8*cos(3*\x r)*sin(\x r)});
\draw[color=blue,line width=.9pt,smooth,domain=-0.523599:-.3] plot ({6*cos(3*\x r)*cos(\x r)},{8*cos(3*\x r)*sin(\x r)});
\draw[color=red,line width=.9pt,smooth,domain=-0.523599:-0.447] plot ({-.055+6*cos(3*\x r)*cos(\x r)},{8*cos(3*\x r)*sin(\x r)});
\draw[->,color=red,line width=.5pt,smooth,domain=-.47:-0.49] plot ({-.055+6*cos(3*\x r)*cos(\x r)},{8*cos(3*\x r)*sin(\x r)});
\draw[color=blue,->,line width=.5pt,smooth,domain=-0.523599:-.33] plot ({6*cos(3*\x r)*cos(\x r)},{-13*cos(3*\x r)*sin(\x r)});
\draw[color=blue,line width=.9pt,smooth,domain=-0.523599:-.2] plot ({6*cos(3*\x r)*cos(\x r)},{-13*cos(3*\x r)*sin(\x r)});
\draw[color=red,line width=.9pt,smooth,domain=-0.523599:-0.462] plot ({-.049+6*cos(3*\x r)*cos(\x r)},{-13*cos(3*\x r)*sin(\x r)});
\draw[->,color=red,line width=.5pt,smooth,domain=-0.523599:-0.49] plot ({-.049+6*cos(3*\x r)*cos(\x r)},{-13*cos(3*\x r)*sin(\x r)});

\draw[color=blue,line width=.9pt,arrow mark=.1,arrow mark=.3,arrow mark=.6,arrow mark=.8,do marks] (4.8533, 2.1316) ..(-1,5) .. (-5,1) .. (-3,-2) .. (-1,0).. (-2.5,1.5) ;
\draw[color=blue,line width=.9pt,arrow mark=.1,arrow mark=.3,arrow mark=.6,arrow mark=.8,do marks]  (-2.5,1.5) ..(-6,-1.3) .. (-7.2,0) .. (-6,1.2) .. (0,-3.5) .. (1.5,-2.5) ;
\draw[color=blue,line width=.9pt,arrow mark=.1,arrow mark=.3,arrow mark=.6,arrow mark=.8,do marks]  (1.5,-2.5) ..(4,-4) .. (7,0) .. (4,5) .. (2,2) ..(3.56308, -1.46959) ;

\draw [color=darkgreen,line width=0.9pt] (0.959167, 1.03923) arc (47.2943:360-32.7599:1.414) ;
\draw [->,color=darkgreen,line width=0.5pt] (0.959167, 1.03923) arc (47.2943:120:1.414) ;

\draw[color=blue] (6.6,4) node[anchor=center]  {\small$\Lambda$};
\draw[color=red] (.75,.4) node[anchor=center]  {\small$\Gamma_1$};
\draw[color=darkgreen] (-.5,1.6) node[anchor=center]  {\small$\Gamma_2$};
\draw[color=red] (.88,-.3) node[anchor=center]  {\small$\Gamma_3$};
\draw[color=black] (-.2,.2) node[anchor=center]  {\small$\alpha$};

\end{tikzpicture}
\caption{Winding number of $\Lambda$.}\label{fig-ne}
\end{center}
\end{figure}
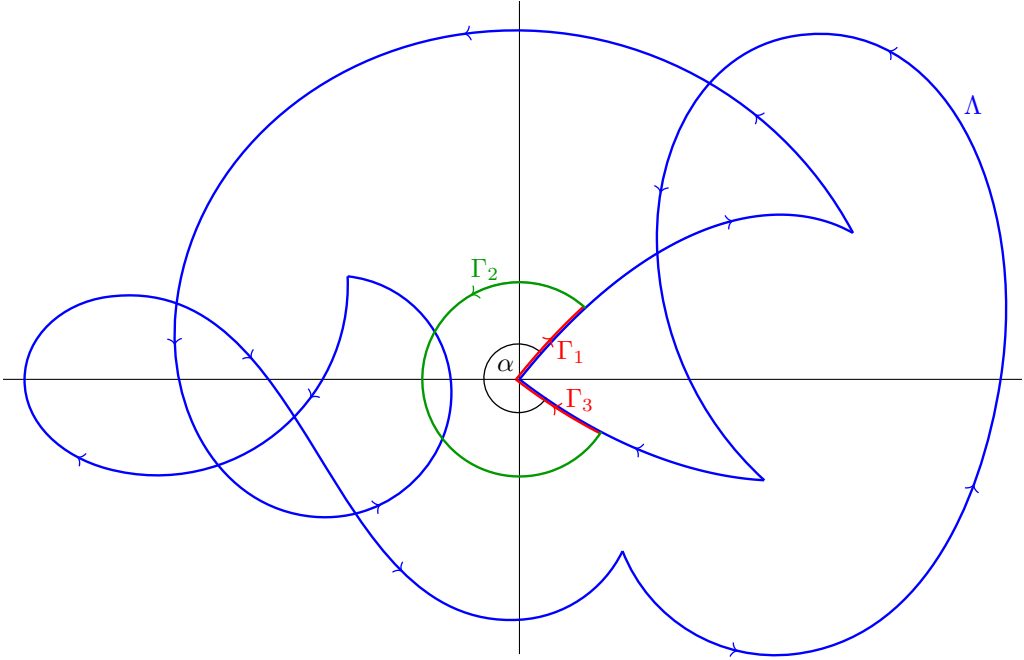Next, we consider an closed piecewise $C^1$ immersion $\Lambda:[a,b]\to \C$ with one zero $x_1$,
and a positively oriented angle $\alpha\in[0,2\pi]$ between $\lim_{t\searrow x_1}\dot\Lambda(t)$ and $-\lim_{t\nearrow x_1}\dot\Lambda(t)$. 
Let  $\Gamma=\Gamma_1+\Gamma_2+\Gamma_3$ be a closed piecewise $C^1$ curve which coincides
with $\Lambda$ in a neighborhood of $x_1$ and which is homotopic in the sense of~(\ref{eq-hom}) to a model sector-curve
with the same angle $\alpha$ (see Figure~\ref{fig-ne}). Then, we decompose $\Lambda$
by $\Lambda=\tilde\Lambda+\Gamma$. By~(\ref{eq-med}) we get
\begin{equation}\label{eq-ap}
n_0(\Lambda)=\operatorname{PV}\frac{1}{2\pi\mathrm i}\oint_\Lambda \frac{\mathrm dz}{z}=
n_0(\tilde\Lambda)+\frac{\alpha}{2\pi}\,.
\end{equation}
%
%
%
Finally, if $\Lambda$ has more than one zero, we obtain in the same way the following proposition:
\begin{prop}\label{prop-winding} 
Let $\Lambda:[a,b]\to \mathbb C$ be a closed piecewise $C^1$ immersion 
and $z_0\in\mathbb C$. Then there exist at most finitely many 
points $x_1,\ldots,x_n\in[a,b]$ such that $\Lambda(x_\ell)=z_0$. Consider a decomposition
$\Lambda=\tilde\Lambda+\Gamma_1+\ldots+\Gamma_n$, where $\tilde\Lambda$ 
coincides with $\Lambda$ outside of small neighborhoods of the points $x_\ell$
and avoids the point $z_0$ by driving around it on small circular arcs in clockwise direction.
The closed curves $\Gamma_\ell$ are homotopic in the sense of~(\ref{eq-hom}) to a model sector-curve
with oriented angle $\alpha_\ell$ between $\lim_{t\searrow x_\ell}\dot\Lambda(t)$ and $-\lim_{t\nearrow x_\ell}\dot\Lambda(t)$
(see Figure~\ref{fig-la}). Then, the winding number of $\Lambda$ with respect to $z_0$ is
$$
n_{z_0}(\Lambda)=\operatorname{PV}\frac{1}{2\pi\mathrm i}\oint_\Lambda \frac{\mathrm dz}{z-z_0}=
n_{z_0}(\tilde\Lambda)+\sum_{\ell=1}^n\frac{\alpha_\ell}{2\pi}\,.
$$
\end{prop}
\begin{figure}[h]
\begin{center}
\begin{tikzpicture}[x=16,y=16]
\clip(-6,-1.7) rectangle (6,3.6);

\draw[color=blue,line width=.9pt,smooth,domain=0.523599:2.61799] plot ({6*cos(3*\x r)*sin(\x r)},{-6*cos(3*\x r)*cos(\x r)});
\draw[-<,color=blue,line width=.5pt,smooth,domain=0.6:.77] plot ({6*cos(3*\x r)*sin(\x r)},{-6*cos(3*\x r)*cos(\x r)});
\draw[-<,color=blue,line width=.5pt,smooth,domain=1.1:1.35] plot ({6*cos(3*\x r)*sin(\x r)},{-6*cos(3*\x r)*cos(\x r)});
\draw[-<,color=blue,line width=.5pt,smooth,domain=1:1.8] plot ({6*cos(3*\x r)*sin(\x r)},{-6*cos(3*\x r)*cos(\x r)});
\draw[-<,color=blue,line width=.5pt,smooth,domain=1.7:2.4] plot ({6*cos(3*\x r)*sin(\x r)},{-6*cos(3*\x r)*cos(\x r)});
\draw[color=blue] (5.4,2.) node[anchor=center]  {\small$\Lambda$};

\draw [color=black,fill=black] (0.0,0) circle (.1);
\draw[color=black] (0,-.4) node[anchor=center]  {\small$z_0$};

\end{tikzpicture}
\begin{tikzpicture}[x=16,y=16]
\clip(-6,-1.7) rectangle (6,3.6);


\draw[color=red,line width=.9pt,smooth,domain=0.523599:0.607826 ] plot ({6*cos(3*\x r)*sin(\x r)},{-6*cos(3*\x r)*cos(\x r)});
\draw[-<,color=red,line width=.5pt,smooth,domain=0.523599:0.57 ] plot ({6*cos(3*\x r)*sin(\x r)},{-6*cos(3*\x r)*cos(\x r)});
\draw[color=red,line width=.9pt,smooth,domain=1.48657:1.65502] plot ({6*cos(3*\x r)*sin(\x r)},{-6*cos(3*\x r)*cos(\x r)});
\draw[-<,color=red,line width=.5pt,smooth,domain=1.48657:1.535] plot ({6*cos(3*\x r)*sin(\x r)},{-6*cos(3*\x r)*cos(\x r)});
\draw[-<,color=red,line width=.5pt,smooth,domain=1.48657:1.62] plot ({6*cos(3*\x r)*sin(\x r)},{-6*cos(3*\x r)*cos(\x r)});
\draw[color=red,line width=.9pt,smooth,domain=2.53377:2.61799 ] plot ({6*cos(3*\x r)*sin(\x r)},{-6*cos(3*\x r)*cos(\x r)});
\draw[-<,color=red,line width=.5pt,smooth,domain=2.53377:2.58 ] plot ({6*cos(3*\x r)*sin(\x r)},{-6*cos(3*\x r)*cos(\x r)});
\draw [color=red,line width=0.9pt] (.95*0.856626, .95*1.23134) arc (55.1742:124.826:.95*1.5) ;
\draw [->,color=red,line width=0.5pt] (.95*0.856626, .95*1.23134) arc (55.1742:110:.95*1.5) ;

\draw [color=red,line width=0.9pt] (-1.49468*.95, 0.126191*.95) arc (184.826:180+360-184.826:1.5*.95) ;
\draw [->,color=red,line width=0.5pt] (-1.49468*.95, 0.126191*.95) arc (184.826:290:1.5*.95) ;

\draw[color=darkgreen,line width=.9pt,smooth,domain=0.606:1.4887 ] plot ({6*cos(3*\x r)*sin(\x r)},{-6*cos(3*\x r)*cos(\x r)});
\draw[color=darkgreen,line width=.9pt,smooth,domain=1.6533:2.5358 ] plot ({6*cos(3*\x r)*sin(\x r)},{-6*cos(3*\x r)*cos(\x r)});
\draw[-<,color=darkgreen,line width=.5pt,smooth,domain=0.7:.77] plot ({6*cos(3*\x r)*sin(\x r)},{-6*cos(3*\x r)*cos(\x r)});
\draw[-<,color=darkgreen,line width=.5pt,smooth,domain=1.3:1.35] plot ({6*cos(3*\x r)*sin(\x r)},{-6*cos(3*\x r)*cos(\x r)});
\draw[-<,color=darkgreen,line width=.5pt,smooth,domain=1.77:1.8] plot ({6*cos(3*\x r)*sin(\x r)},{-6*cos(3*\x r)*cos(\x r)});
\draw[-<,color=darkgreen,line width=.5pt,smooth,domain=2.33:2.4] plot ({6*cos(3*\x r)*sin(\x r)},{-6*cos(3*\x r)*cos(\x r)});
\draw[color=darkgreen] (5.45,2.) node[anchor=center]  {\small$\tilde\Lambda$};
\draw [color=darkgreen,line width=0.9pt] (0.856626, 1.23134) arc (55.1742:124.826:1.5) ;
\draw [-<,color=darkgreen,line width=0.5pt] (0.856626, 1.23134) arc (55.1742:80:1.5) ;
\draw [color=darkgreen,line width=0.9pt] (-1.49468, 0.126191) arc (184.826:180+360-184.826:1.5) ;
\draw [-<,color=darkgreen,line width=0.5pt] (-1.49468, 0.126191) arc (184.826:260:1.5) ;

\draw [color=white,fill=white] (0.0,1.05) circle (.3);

\draw[color=red] (.0,1) node[anchor=center]  {\small$\Gamma_1$};
\draw[color=red] (.5,-.5) node[anchor=center]  {\small$\Gamma_2$};
\draw [color=black,fill=black] (0.0,0) circle (.1);
\draw[color=black] (-.3,-.4) node[anchor=center]  {\small$z_0$};

\end{tikzpicture}

\caption{Decomposition of $\Lambda=\tilde\Lambda+\Gamma_1+\Gamma_2$.}\label{fig-la}
\end{center}
\end{figure}
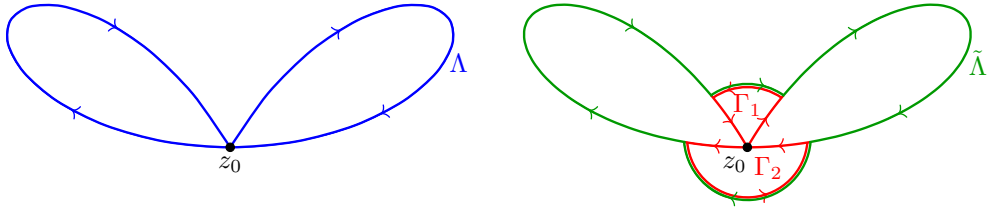
\begin{proof}
First we show that for only finitely many points $x_\ell$, we have $\Lambda(x_\ell)=z_0$.
It suffices to consider a $C^1$ curve $\Lambda:[a,b]\to \mathbb R^2$
parametrized by arc length,
and $z_0=(0,0)$. Assume by contradiction that $\Lambda$ has  infinitely many zeros $x_\ell$. Then 
there is a subsequence, again denoted by $x_\ell$, which converges to a point $x\in[a,b]$,
and we may assume, that $x_\ell$ is increasing.
Then, by Rolle's Theorem, since $\Lambda_1(x_\ell)=\Lambda_2(x_\ell)=0$, there are
points $u_\ell,v_\ell\in(x_\ell,x_{\ell+1})$ such that $\Lambda'_1(u_\ell)=\Lambda'_2(v_\ell)=0$.
But then, $\Lambda'_1(v_\ell)=\Lambda'_2(u_\ell)=1$. Hence, $\Lambda'$ cannot be continuous.

For the rest of the proof, observe, that $\tilde\Lambda$ avoids the point $z_0$. Thus, we have
\begin{eqnarray*}
n_{z_0}(\Lambda)&=&\operatorname{PV}\frac{1}{2\pi\mathrm i}\oint_\Lambda \frac{\mathrm dz}{z-z_0}=\\
&=&\frac{1}{2\pi\mathrm i}\oint_{\tilde\Lambda} \frac{\mathrm dz}{z-z_0}
+\sum_{\ell=1}^n\operatorname{PV}\frac{1}{2\pi\mathrm i}\oint_{\Gamma_\ell} \frac{\mathrm dz}{z-z_0}=\\
&=&n_{z_0}(\tilde\Lambda)+\sum_{\ell=1}^n\frac{\alpha_\ell}{2\pi}\,,
\end{eqnarray*}
where we have used~(\ref{eq-med}) in the last step.
\end{proof}
Proposition~\ref{prop-winding} generalizes immediately from curves
to cycles.
The Definition~\ref{def-winding} of a generalized winding number turns out to be useful
as it allows to generalize the residue theorem (see Theorem \ref{thm-residue} below). But before
we turn our attention this subject, let us briefly reformulate the 
formula~(\ref{windungszahl}) for the generalized winding number
as an integral in the real plane. Interestingly, while the winding number as a
complex integral requires an interpretation as a principal value,
the real counterpart turns out to have a bounded integrand.


If $\Lambda = x + \mathrm i y:[a,b]\to\C$ is a closed piecewise $C^1$ curve, then $\mathrm dz/z$ decomposes as
$$
\frac{\mathrm dz}{z}=
\frac{x\dot x + y \dot y}{x^2+y^2}\,\mathrm dt + \mathrm i \frac{x\dot y-y\dot x}{x^2+y^2}\,\mathrm dt.
$$
The considerations above imply that if $\Lambda$ is a closed piecewise $C^1$ immersion, then
$$
\operatorname{PV}\oint_\Lambda \frac{x\,\mathrm dx + y \,\mathrm dy}{x^2+y^2}= 0,
$$
hence only the imaginary part of $\mathrm dz/z$ is relevant for computing $n_0(\Lambda)$. We have the following proposition regarding its regularity:
\begin{prop}\label{winding}
Let $\Lambda=x+\mathrm i y:[a,b]\to \C$ be a closed piecewise $C^{1,1}$ immersion. Then
$$n_0(\Lambda)=\frac{1}{2\pi}\int_{a}^b\frac{x(t)\dot y(t)- y(t)\dot x(t)}{x(t)^2+y(t)^2}\,\mathrm dt$$
and the corresponding integrand is bounded. If $\Lambda$ is $C^2$ in a neighbourhood of a point $\tilde t\in (a,b)$ with $\Lambda(\tilde t)=0$, then
$$
\lim_{t\to \tilde t}\frac{x(t)\dot y(t)- y(t)\dot x(t)}{x(t)^2+y(t)^2} = \frac{1}{2}k_\Lambda(\tilde t)|\dot \Lambda(\tilde t)|,
$$
where $$k_\Lambda(\tilde t)=\frac{\dot x(\tilde t)\ddot y(\tilde t)- \dot y(\tilde t)\ddot x(\tilde t)}{(\dot x(\tilde t)^2+\dot y(\tilde t)^2)^{3/2}} $$ is the signed curvature of $\Lambda$ in $\tilde t$.
\end{prop}
\begin{proof}
Let $\Lambda$ be a closed piecewise $C^{1,1}$ immersion. If $\Lambda$ avoids the origin, then the
integrand is obviously bounded. On the other hand, as in the proof of Proposition~\ref{prop-winding},
it follows that $\Lambda$ can have at most finitely many
zeros, say $a<t_0<t_1< \ldots <t_n<b$. It suffices to concentrate on one of the zeros $t_\ell$ of $\Lambda$, and to simplify the
notation we assume $\tilde t=0$ for the rest of the proof.
We first show that the integrand is bounded provided $\Lambda$ is $C^{1,1}$ in a neighbourhood $U=(-\varepsilon,\varepsilon)$ of $0$. 
In this case, $\dot x$ and $\dot y$ are Lipschitz functions on $U$
and 
\begin{align}\label{h1}
|x(t)\dot y(t)-y(t)\dot x(t)|&=\left|
\int_0^t \dot x(s)\,\mathrm ds \dot y(t)-
\int_0^t \dot y(s)\,\mathrm ds \dot x(t)\right|\nonumber\\
&\leq\int_0^t \big|\!
\bigl(\dot x(s) -
\dot x(t)\bigr)\dot y(t)+
\dot x(t)\bigl(\dot y(t)-
\dot y(s)\bigr)\!
\big|\,\mathrm ds\nonumber\\
& \leq C \int_0^t\big(|s-t||\dot y(t)| + |\dot x(t)||s-t|\big)\,\mathrm ds=O(t^2).
\end{align}
On the other hand
\begin{equation}\label{h2}
x(t)^2+y(t)^2 =\bigl(t\dot x(0)+o(t)\bigr)^2+\bigl(t\dot y(0)+o(t)\bigr)^2=
t^2\bigl(\dot x(0)^2+\dot y(0)^2\bigr)+o(t^2).
\end{equation}
Together, (\ref{h1}) and (\ref{h2}), imply that the integrand in Proposition~\ref{winding} is bounded in $U$.

If $\Lambda$ is only $C^{1,1}$ on $U^-=(-\varepsilon,0]$ and $U^+=[0,\varepsilon)$, the claim follows in the same way
by considering the unilateral intervals left and right of the zero.


Now we assume that $\Lambda$ is $C^2$ in a neighbourhood $U=(-\varepsilon,\varepsilon)$ of the zero $\tilde t=0$.
It remains to show that the limit 
$$
\lim_{t\to 0}\frac{x(t)\dot y(t)- y(t)\dot x(t)}{x(t)^2+y(t)^2}
$$
has the geometrical interpretation stated in the proposition. In fact, if $\Lambda(0)=0$
we find
$$\begin{aligned}
x(t)\dot y(t)& = \left(t\dot x(0)+\frac{t^2}{2}\ddot x(0)+o(t^2)\right)(\dot y(0)+t\ddot y(0)+o(t))\\
& = t\dot x(0)\dot y(0)+t^2\left(\dot x(0)\ddot y(0)+\frac{1}{2}\ddot x(0)\dot y(0)\right)+o(t^2)
\end{aligned}$$
and hence
\begin{equation}\label{h3}
\begin{aligned}
x(t)\dot y(t)-\dot x(t) y(t) & = t^2\left(\dot x(0)\ddot y(0)-\dot y(0)\ddot x(0)+\frac{1}{2}\left(\ddot x(0)\dot y(0)-\ddot y(0)\dot x(0)\right)\right)+o(t^2)\\
& = \frac{t^2}{2}\Big(\dot x(0)\ddot y(0)-\dot y(0)\ddot x(0)\Big)+o(t^2).
\end{aligned}
\end{equation}
On the other hand
\begin{equation}\label{h4}\begin{aligned}
x^2(t)+y^2(t) & = \left(t\dot x(0)+\frac{t^2}{2}\ddot x(0) + o(t^2)\right)^2+\left(t\dot y(0)+\frac{t^2}{2}\ddot y(0) + o(t^2)\right)^2\\
& = t^2(\dot x(0)^2+\dot y(0)^2)+t^3(\dot x(0)\ddot x(0)+\dot y(0)\ddot y(0)) + o(t^3).
\end{aligned}\end{equation}
From \eqref{h3} and \eqref{h4} we deduce
$$
\lim_{t\to \tilde t}\frac{x(t)\dot y(t)- y(t)\dot x(t)}{x(t)^2+y(t)^2} = \frac{1}{2}k_\Lambda(\tilde t)|\dot \Lambda(\tilde t)|.
$$
\end{proof}
It is worth noticing, that Proposition~\ref{winding} is more than just a technical
remark. We will see in Section~\ref{appl} an application of the observation
that the imaginary part of the integrand is bounded.

The geometrical meaning of the winding number can be used to characterize the topological phases in one-dimensional chiral non-Hermitian systems. Chiral symmetry ensures that the winding number of Hermitian systems are integers, but
non-Hermitian systems  can take half integer values: see \cite{physik} for the corresponding physical interpretation of Proposition~\ref{winding}.

\begin{example}\label{zeppelin}
Consider the curve $\Lambda:[0,2\pi]\to \C$ given by
$$\Lambda(t)=x(t)+\mathrm iy(t):= \cos(t) + \cos(2t)+\mathrm i\sin (2t)$$
which passes through the origin at $t=\pi$ (see Figure~\ref{fig-6}). According to Proposition~\ref{winding},
$$
n_0(\Lambda)=\frac{1}{2\pi}\int_{a}^b\frac{x(t)\dot y(t)- y(t)\dot x(t)}{x(t)^2+y(t)^2}\,\mathrm dt = \frac{3}{2}
$$
and the corresponding integrand is continuous.
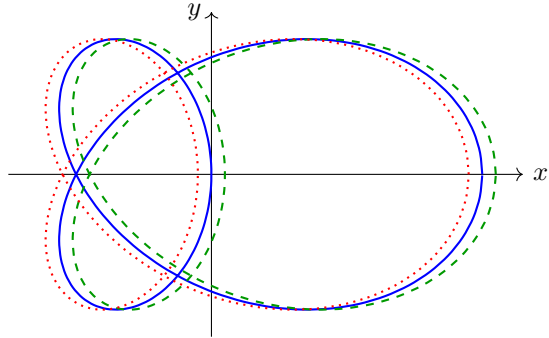
\begin{figure}[h]
\begin{center}
\begin{tikzpicture}[scale=1.8]
\draw[->] (-1.5,0) -- (2.3,0)  node[anchor=west] {\small$x$};
\draw[->] (0,-1.2) -- (0,1.2)  node[anchor=east] {\small$y$};
\draw[thick,smooth,color=blue,samples=80,domain=0:6.28] plot ( {cos(\x r)+cos(2*\x r)} , {sin(2*\x r)} );
\draw[thick,dashed,smooth,color=darkgreen,samples=80,domain=0:6.28] plot ( {.1+cos(\x r)+cos(2*\x r)} , {sin(2*\x r)} );
\draw[thick,dotted,smooth,color=red,samples=80,domain=0:6.28] plot ( {cos(\x r)+cos(2*\x r)-.1} , {sin(2*\x r)} );
\end{tikzpicture}
\caption{The (solid) curve from Example \ref{zeppelin} has winding number $\frac{3}{2}$ with
respect to the origin, and the dashed and dotted one have winding number 2 and 1 respectively.}\label{fig-6}
\end{center}
\end{figure}
\end{example}

\section{A Generalized Residue Theorem}
Let $U\subset \C$ be an open neighborhood of zero and let $f$ be a holomorphic
function on $U\setminus\{0\}$. Then there exists a Laurent series which represents $f$   
in a punctured neighborhood $\{z\in\mathbb C: 0<|z|<R\}$ of zero:
$$
f(z) = \underbrace{\ldots + \frac{a_{-1}}{z}}_{=g(z)}+\underbrace{a_0+a_1z+\ldots}_{=h(z)}.
$$
For a closed piecewise  $C^1$ curve $\gamma$ with $|\gamma|<R$, we have 
$$
\operatorname{PV}\oint_{\gamma}f(z)\,\mathrm dz = \operatorname{PV}\oint_{\gamma}(g(z)+h(z))\,\mathrm dz = \operatorname{PV}\oint_{\gamma}g(z)\,\mathrm dz
$$
by the Cauchy integral theorem, provided the principal value exists. 
If $f$ has only a pole of first order in $0$, then the discussion in Section~\ref{generalizedwindingnumber}
shows, that the principal value indeed exists. The general case however is more
delicate: let us
first consider a model sector curve $\gamma$ with angle $\alpha\in[0,2\pi]$, and let $n>1$. Then we have
\begin{align}
\operatorname{PV}\oint_{\gamma}\frac{\mathrm dz}{z^n} & =\lim_{\varepsilon\searrow 0}\left(\int_\varepsilon^r\frac{\mathrm dt}{t^n}+\int_0^\alpha \frac{r\mathrm i\mathrm e^{it}}{r^n\mathrm e^{\mathrm int}}\,\mathrm dt + \int_0^{r-\varepsilon}\frac{-\mathrm e^{i\alpha}}{(r-t)^n\mathrm e^{\mathrm in\alpha}}\,\mathrm dt\right)\nonumber\\
& = \lim_{\varepsilon\searrow 0}\left(\frac{r^n\varepsilon-r\varepsilon^n}{(r\varepsilon)^n(n-1)}-
\frac{\mathrm e^{-\alpha(n-1)\mathrm i}-1}{r^{n-1}(n-1)}+\frac{r\varepsilon^n-r^n\varepsilon}{(r\varepsilon)^n(n-1)}\mathrm e^{-\alpha(n-1)\mathrm i}\right)
\nonumber\\
& = \lim_{\varepsilon\searrow 0}\frac{1-\mathrm e^{-\mathrm i(n-1)\alpha}}{(n-1)\varepsilon^{n-1}}=
\begin{cases}
0&\text{if $\frac{\alpha(n-1)}{2\pi}\in\Z$,}\\
\text{complex infinity}&\text{otherwise.}
\end{cases}\label{ganz}
\end{align}
On an intuitive level it is clear that an angle condition decides whether the 
limit exists or not: Indeed, in order to compensate the purely real values
on $\gamma_1$ (see Figure~\ref{fig-gamma}), the integral 
along $\gamma_3$ cannot have a  non-real singular part.
Hence, the principal value in~(\ref{ganz}) exists (and is actually $0$) if and only if 
$$
\alpha = \frac{2k\pi}{n-1}
$$
for some $k\in\Z$. Stated differently: If $\alpha = \frac{p}{q}\pi$ for some $p,q\in\mathbb N$, $q\neq 0$,
then 
\begin{equation}\label{principalvalue}
\operatorname{PV}\oint_{\gamma}\frac{\mathrm dz}{z^n}=0
\end{equation}
if $n=\frac{2kq}{p}+1$ for an integer $k\ge 0$, otherwise the principal value \eqref{principalvalue} is infinite. 
Therefore we obtain:
\begin{lemma}\label{lem-lem}
Let $\alpha = \frac{p}{q}\pi$ for some $p,q\in\mathbb N$, $q\neq 0$.
If the Laurent series of $f$ only contains terms $a_n/z^n$ for indices of the form $n=\frac{2kq}{p}+1$
for integers $k\ge 0$, and if $\gamma$ is a model sector-curve with angle $\alpha$  and radius smaller than the
radius of convergence of the Laurent series, then there holds
\begin{equation}\label{con}
\operatorname{PV}\frac1{2\pi\mathrm i}\oint_{\gamma} f(z)\,\mathrm dz = n_0(\gamma)\operatorname{res}_0f(z).
\end{equation}
\end{lemma}
\begin{proof}
If $f$ has a pole in $0$, then (\ref{con}) follows directly from~(\ref{principalvalue}).
If $0$ is an essential singularity of $f=\sum_{k\in\mathbb Z}a_kz^k$, we observe that
the Laurent series $f_n(z)=\sum_{k=-n}^\infty a_kz^k$ converges locally uniformly to $f(z)$.
Then we have for $\varepsilon>0$
\begin{multline*}
\frac1{2\pi\mathrm i}\int_{|\gamma|>\varepsilon} f(z)\,\mathrm dz = 
\frac1{2\pi\mathrm i}\int_{|\gamma|>\varepsilon} \bigl(f(z)-f_n(z)\bigr)\,\mathrm dz +\\+
\Bigl(\frac1{2\pi\mathrm i}\int_{|\gamma|>\varepsilon} f_n(z)\,\mathrm dz-
n_0(\gamma)\operatorname{res}_0f(z)\Bigr)+n_0(\gamma)\operatorname{res}_0f(z)=:I+II+III.
\end{multline*}
Now, for $\delta>0$, we choose $\varepsilon>0$ small enough, such that the absolute value of the second term $II$ on the right
is smaller than $\delta$. Note that by~(\ref{ganz}) this choice does not depend on $n$.
Then we can choose $n$, depending on $\varepsilon$, 
large enough such that the absolute value of the first term $I$ is also smaller than $\delta$,
and the claim follows.
\end{proof}
For a more general curve than a model sector curve, we need the following definition:
\begin{definition}
Let $\Gamma:(a,b)\to\mathbb C$ be a piecewise $C^1$ curve and $\Gamma(x_1)=:z_1$.
Let $t^+$ and $t^-$ be the tangents in $z_1$ in the direction
$\lim_{x\searrow x_1}\dot\Gamma(x)$ and $-\lim_{x\nearrow x_1}\dot\Gamma(x)$ respectively.
We say that $\Gamma$ is flat of order $n$ in $x_1$, if 
\begin{eqnarray*}
&&|\Gamma(x)-P^+(\Gamma(x))|=o(|\Gamma(x)-z_1|^n)\text{ for $x\searrow x_1$ and}\\
&&|\Gamma(x)-P^-(\Gamma(x))|=o(|\Gamma(x)-z_1|^n)\text{ for $x\nearrow x_1$}
\end{eqnarray*}
where $P^+$ and $P^-$ denote the orthogonal projection to $t^+$ and $t^-$ respectively (see Figure~\ref{fig-flat}).
\end{definition}
Notice, that a piecewise $C^1$ curve is always flat of order 1 in all of its points.
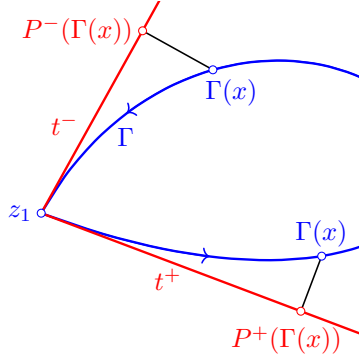
\begin{figure}[h]
\begin{center}
\begin{tikzpicture}[x=160,y=160]
\clip(-.1,-.4) rectangle (.75,.5);

\draw[color=blue,line width=.9pt,smooth,domain=-.5236:.5236] plot ({.9*cos(\x r)*cos(3*\x r)-.7*cos(3*\x r)*sin(\x r)}, {.2*cos(\x r)*cos(3*\x r) +1.2*cos(3*\x r)*sin(\x r)});
\draw[->,color=blue,line width=.6pt,smooth,domain=-.5236:-.4] plot ({.9*cos(\x r)*cos(3*\x r)-.7*cos(3*\x r)*sin(\x r)}, {.2*cos(\x r)*cos(3*\x r) +1.2*cos(3*\x r)*sin(\x r)});
\draw[->,color=blue,line width=.6pt,smooth,domain=0:.4] plot ({.9*cos(\x r)*cos(3*\x r)-.7*cos(3*\x r)*sin(\x r)}, {.2*cos(\x r)*cos(3*\x r) +1.2*cos(3*\x r)*sin(\x r)});

\draw[color=red,line width=.9pt,smooth,domain=0:0.523599] plot ({3.38827*\x}, {-1.28038*\x});
\draw[color=red,line width=.9pt,smooth,domain=0:0.523599] plot ({1.28827*\x}, {2.31962*\x});

\draw[color=black,line width=.6pt] (0.663051, -0.101669) to (0.613818, -0.231954);
\draw [color=blue,fill=white] (0.663051, -0.101669) circle (1.5pt);
\draw [color=red,fill=white] (0.613818, -0.231954) circle (1.5pt);

\draw[color=black,line width=.6pt] (0.405873, 0.339207) to (0.239658, 0.43152);
\draw [color=blue,fill=white] (0.405873, 0.339207) circle (1.5pt);
\draw [color=red,fill=white] (0.239658, 0.43152) circle (1.5pt);
\draw [color=blue,fill=white] (0, 0) circle (1.5pt);

\begin{small}
\draw[color=blue] (.2,.18) node[anchor=center]  {\small$\Gamma$};
\draw[color=blue] (0.663051, -0.101669) node[anchor=south]  {\small$\Gamma(x)$};
\draw[color=blue] (0.45, 0.339207) node[anchor=north]  {\small$\Gamma(x)$};

\draw[color=red] (0.58, -0.245) node[anchor=north]  {\small$P^+(\Gamma(x))$};
\draw[color=red] (0.239658, 0.43152) node[anchor=east]  {\small$P^-(\Gamma(x))$};

\draw[color=red] (0.11, 0.21) node[anchor=east]  {\small$t^-$};
\draw[color=red] (0.3, -0.11) node[anchor=north]  {\small$t^+$};
\draw[color=blue] (0, -0) node[anchor=east]  {\small$z_1$};
\end{small}

\end{tikzpicture}
\caption{$\Gamma$ is flat of order $n$ in $x_1$.}\label{fig-flat}
\end{center}
\end{figure}

Now, let us consider a closed piecewise $C^1$ immersion $\Gamma:[a,b]\to\C$ starting and ending in $0$ but
such that $\Gamma(t)\ne 0$ for all $t\ne a,b$ 
and such that the (positively oriented) angle between $\lim_{t\searrow a}\dot\Gamma(t)$ and $-\lim_{t\nearrow b}\dot\Gamma(t)$ equals $\alpha\in[0,2\pi]$. 
We assume that, after  a suitable rotation, $\lim_{t\searrow a}\dot\Gamma(t)$ is a positive
real number and that $\Gamma$  is homotopic in the sense of~(\ref{eq-hom}) to a model sector-curve $\gamma$
with the same angle $\alpha$. Moreover, we assume that $\Gamma$ is flat of order $n$ in $0$.
Then, as in Section~\ref{generalizedwindingnumber} (see Figure~\ref{fig-cauchy}), we have for $n>1$
$$\begin{aligned}
\left|\int_{|\Gamma|>\varepsilon}\frac{\mathrm dz}{z^n}-\int_{|\gamma|>\varepsilon}\frac{\mathrm dz}{z^n} \right|& = \left| \int_{\beta}\frac{\mathrm dz}{z^n}-\int_{|\gamma|>\varepsilon}\frac{\mathrm dz}{z^n}\right|=\\
& = \left|\int_{\alpha_1+\alpha_2}\frac{\mathrm dz}{z^n}\right|\\
 &\leq \frac{1}{\varepsilon^n}\underbrace{\operatorname{Length}(\alpha_1+\alpha_2)}_{=o(\varepsilon^n)}
\longrightarrow 0 \text{ for $\varepsilon\searrow 0$.}
\end{aligned}$$
Hence, we only have a finite principal value if $\alpha = \frac{p}{q}\pi$ for some $p,q\in\mathbb N$, $q\neq 0$, and
\begin{equation}\label{34}
\operatorname{PV}\oint_{\Gamma}\frac{\mathrm dz}{z^n}=0
\end{equation}
if $n=\frac{2kq}{p}+1$ for an integer $k\ge 0$, otherwise the principal value is infinite. This leads
to the main Theorem:
\begin{theorem}\label{thm-residue}
Let $U\subset \C$ be an open set, and $S=\{z_1,z_2,\ldots\}\subset U$ be a set without
accumulation points in $U$ such that $f:U\setminus S\to\C$ is holomorphic.
Moreover, let $C$ be a null-homologous immersed piecewise 
 $C^1$ cycle in $U$ such that $C$ only contains singularities of $f$ which are poles of order 1. Then
\begin{equation}\label{he0}
\operatorname{PV}\frac{1}{2\pi\mathrm i}\oint_C f(z)\,\mathrm dz = \sum_\ell n_{z_\ell}(C)\operatorname{res}_{z_\ell}f(z).
\end{equation}
The formula remains correct for poles of higher order on $C$ if
the following two conditions hold:\renewcommand{\labelenumi}{(\theenumi)}
\renewcommand{\theenumi}{\Alph{enumi}}%
\begin{enumerate}
\item If $z_0$ is a pole on $C$ of order $n$, then $C$ is flat of order $n$ in $z_0$, or,
if $z_0$ is an essential singularity, $C$ coincides near $z_0$ locally with the tangents $t^+$ and $t^-$ in $z_0$.\label{cond1}
\item If $z_0$ is a singularity of $f$ on $C$ and $\alpha$ is the angle between the tangents
$t^+$ and $t^-$ in $z_0$, then
$\alpha=\frac pq\pi$, $p,q\in\mathbb N, q\neq 0$, and the Laurent series of $f$ in $z_0$ contains only
terms $a_n/(z-z_0)^n$ with $a_n\neq 0$ for indices of the form $n=\frac{2kq}p+1$, $k\ge 0$ an integer.\label{cond2}
\end{enumerate}
\end{theorem}
\begin{proof}
Let $C=\sum_{\ell=1}^k m_\ell \gamma_\ell$ with $m_\ell\in\mathbb Z$ and where $\gamma_\ell:[0,1]\to\mathbb C$
are closed piecewise $C^1$ immersions. Then, there are at most finitely many
points $x_{\ell1},\ldots,x_{\ell k_\ell}$ such that $\gamma_\ell(x_{\ell j})=z_{\ell j}\in S$. 
For each $\ell$ consider a decomposition
$\gamma_\ell = \tilde\gamma_\ell + \Gamma_{\ell 1}+\ldots + \Gamma_{\ell k_\ell}$, where
$\tilde\gamma_\ell$ coincides with $\gamma_\ell$ outside of small
neighborhoods of  the points $x_{\ell j}$ and avoids the singularity at $\gamma_\ell({x_{\ell j}})$
by driving around it on small circular arcs in clockwise direction. The closed curves
$\Gamma_{\ell j}$ are homotopic in the sense of~(\ref{eq-hom}) to a model sector-curve
with oriented angle $\alpha_{\ell j}$ between the tangents $\lim_{t\searrow x_{\ell j}}\dot\gamma_\ell$ and 
$-\lim_{t\nearrow x_{\ell j}}\dot\gamma_\ell$. The circular arcs are chosen small enough
such that no singularity lies in the interior of the sectors whose boundary are the curves $\Gamma_{\ell j}$,
and such that these sectors are contained in $U$.
Observe that the cycle $\tilde C:=\sum_{\ell=1}^k m_\ell \tilde\gamma_\ell$
avoids the singularities of $f$ and is null-homologous in $U$. Hence, in the sequel we may apply
the classical residue theorem to $\tilde C$.

Now, suppose that the two conditions~(\ref{cond1}) and~(\ref{cond2}) hold. This
covers in particular the case when only poles of first order lie on $C$.
Then, we have, by the classical residue theorem applied with $\tilde C$,
by Lemma~\ref{lem-lem}, and~(\ref{34})
\begin{align}
\operatorname{PV}\frac1{2\pi\mathrm i}\oint_C f(z)\,\mathrm dz &=
\operatorname{PV}\frac1{2\pi\mathrm i}\oint_{\tilde C} f(z)\,\mathrm dz+
\sum_{\ell=1}^k m_\ell \sum_{j=1}^{k_\ell}\operatorname{PV}\frac1{2\pi\mathrm i} \oint_{\Gamma_{\ell j}} f(z)\,\mathrm dz=\nonumber\\
 &=\sum_{z\in S} n_{z}(\tilde C)\operatorname{res}_{z}f(z)+
 \sum_{\ell=1}^k m_\ell \sum_{j=1}^{k_\ell} n_{z_{\ell j}}(\Gamma_{\ell j})\operatorname{res}_{z_{\ell j}}f(z).\label{he1}
\end{align}
The first sum in~(\ref{he1}) runs over 
\renewcommand{\labelenumi}{(\theenumi)}
\renewcommand{\theenumi}{\Roman{enumi}}%
\begin{enumerate}
\item the singularities which are not lying on $C$, with winding number $\neq 0$\label{ke1}.
\item the singularities on $C$.\label{ke2}
\end{enumerate}
Thus, the summands in~(\ref{ke1}) appear exactly also in the sum in~(\ref{he0}) since
for singularities $z$ not on $C$ we have $n_{z}(C)=n_{z}(\tilde C)$. 
The summands in~(\ref{ke2})
coming from a singularity $z_{\ell j}$ on $C$ together with the corresponding terms 
in the double sum in~(\ref{he1}) give
$$
\Bigl(n_{z_{\ell j}}(\tilde C) +
 \sum_{\ell=1}^k m_\ell \sum_{j=1}^{k_\ell}n_{z_{\ell j}}(\Gamma_{\ell j})
\Bigr)\operatorname{res}_{z_{\ell j}}f(z)=n_{z_{\ell j}}(C)\operatorname{res}_{z_{\ell j}}f(z)
$$
and we are done.
\end{proof}
As corollaries of Theorem~\ref{thm-residue} we obtain the residue
theorems~\cite[Theorem 1]{legua} and~\cite[Theorem 4.8f]{henrici}.

\subsection{Application}\label{appl}
In~\cite{legua}, the version of the residue theorem is used to calculate
principal values of integrals. At first sight it seems that this
is the only advantage of Theorem~\ref{thm-residue} over the
classical residue theorem. After all, poles on the curve $\gamma$
necessarily mean that one is forced to consider principal values.
However, Proposition~\ref{winding} shows that it is possible
to use Theorem~\ref{thm-residue} to compute integrals with
bounded integrand. 
\begin{example}
We want to compute the integral
$$
\int_0^\infty \frac{\operatorname{sinc}(t) \sinh (t)}{\cos (t)+\cosh (t)}\,\mathrm dt.
$$
The current computer algebra systems give up on this integral after
giving it some thought. To determine the integral we interpret the integrand as follows
$$
\frac{\operatorname{sinc}(t) \sinh (t)}{ \cos (t)+\cosh (t)}=
\operatorname{im}f(\gamma_3(t))\dot\gamma_3(t)
$$
for
$$
f(z)=-\frac{\cos(z/2)}{z\cosh(z/2)}
$$
and $\gamma = \gamma_1 + \gamma_2 - \gamma_3$, where\arraycolsep3pt
$$\begin{array}{lrll}
\gamma_1:& [0,r]&\to\C,                                &~ t\mapsto t-\mathrm it,\\
\gamma_2:&[-\frac{\pi}{4},\frac{\pi}4]&\to\C,  &~ t\mapsto \sqrt 2r\mathrm e^{\mathrm it},\\
\gamma_3:&[0,r]&\to\C,                                 &~ t\mapsto t+\mathrm it.
\end{array}$$
Note that $f$ has a pole of order $1$ in $z=0$ on $\gamma$ with residue $-1$. The winding number of
$\gamma$ with respect to $z=0$ is $\frac14$. 
By Theorem~\ref{thm-residue} we get:
$$
\operatorname{im}\oint_\gamma f(z)\,\mathrm dz=
\operatorname{im}\Bigl(\int_{\gamma_1} f(z)\,\mathrm dz+\int_{\gamma_2} f(z)\,\mathrm dz-\int_{\gamma_3} f(z)\,\mathrm dz\Bigr)
=\operatorname{im}\bigl(\frac14\cdot (-1)\cdot 2\pi\mathrm i\bigr)=-\frac\pi2.
$$
The integral along $\gamma_2$ converges to $0$ as $r$ tends to infinity, and
$$
\operatorname{im}\int_{\gamma_1} f(z)\,\mathrm dz=-\operatorname{im}\int_{\gamma_3} f(z)\,\mathrm dz=
-\int_0^r \frac{\operatorname{sinc}(t) \sinh (t)}{ \cos (t)+\cosh (t)}\,\mathrm dt.
$$
Hence, we find the value of this improper integral:
$$
\int_0^\infty \frac{\operatorname{sinc} (t) \sinh (t)}{ \cos (t)+\cosh (t)}\,\mathrm dt=\frac\pi4\,.
$$
\end{example}

\subsection{Connection to the Sokhotski\u{\i}-Plemelj Theorem}
In this section, we want to briefly show that the above-mentioned version of the residue theorem in~\cite[Theorem 1]{legua} 
can be obtained as a corollary of the Sokhotski\u{\i}-Plemelj Theorem.

Let $U\subset \mathbb C$ be an open set, let $S=\{z_1,\ldots\}\subset U$ be a set without accumulation points and let $f:U\setminus S\to \mathbb C$ be a holomorphic function. Let furthermore $D$ be a bounded domain with piecewise $C^1$-boundary $C$, consisting of finitely many components, such that $\bar D\subset U$. As usual, we assume that $C$ is oriented such that $D$ lies always on the left with respect to the direction
of the parametrization. If $f$ only has first order poles on $C$, we may use a decomposition
$$
f = g + \sum_{k=1}^m f_k,
$$
where $f_k$ is holomorphic away from a single first order pole $z_k\in C$ and $g$ has only singularities in $z_{m+1}, z_{m+2}, \ldots\notin C$. Then $\varphi_k(z) := f_k(z)(z-z_k)$ is holomorphic.
Let
$$
F_k(z) := \frac{1}{2\pi \mathrm i} \int_C \frac{\varphi_k(\xi)}{\xi-z}\,\mathrm d\xi.
$$
By Cauchy's integral formula we have $F_k(z) = \varphi_k(z)$ provided $z\in D$. Furthermore
$$\begin{aligned}
F^+_k(z_k) & := \lim_{D\ni z \to z_k}F_k(z) = \lim_{D\ni z \to z_k}\varphi_k(z) = \lim_{D\ni z \to z_k}f_k(z)(z-z_k)\\
&\textcolor{white}{:}=\operatorname{res}_{z_k}f_k(z)=\operatorname{res}_{z_k}f(z)
\end{aligned}
$$
and
$$
F_k(z_k) = \operatorname{PV}\frac{1}{2\pi \mathrm i} \int_C \frac{\varphi_k(\xi)}{\xi-z_k}\,\mathrm d\xi= \operatorname{PV}\frac{1}{2\pi \mathrm i} \int_C f_k(\xi)\,\mathrm d\xi.
$$
According to the Sokhotski\u{\i}-Plemelj formula (see \cite[p. 385, (3)]{hazewinkel} or 
\cite[Chapter 3]{russisch}) we find
$$
F^+_k(z_k) = F_k(z_k) + \left(1-\frac{\alpha_k}{2\pi}\right)\varphi_k(z_k),
$$
where $\alpha_k$ is the interior angle of $C$ in $z_k$ and hence
$$
\operatorname{res}_{z_k}f(z) = \operatorname{PV}\frac{1}{2\pi \mathrm i}\int_C f_k(\xi)\,\mathrm d \xi + \left(1-\frac{\alpha_k}{2\pi}\right)\operatorname{res}_{z_k}f(z)
$$
and after rearranging
$$
\operatorname{PV}\frac{1}{2\pi \mathrm i}\int_C f_k(\xi)\,\mathrm d\xi = \frac{\alpha_k}{2\pi} \operatorname{res}_{z_k}f(z).
$$
Therefore we find
$$\begin{aligned}
\operatorname{PV}\frac{1}{2\pi \mathrm i}\int_C f(\xi)\,\mathrm d\xi  & = \operatorname{PV}\frac{1}{2\pi \mathrm i}\int_C \left(g(\xi)+\sum_{k=1}^m f_k(\xi)\right)\,\mathrm d\xi\\
& = \sum_{k\geqslant m+1}\operatorname{res}_{z_k}f(z) + \sum_{k=1}^m \frac{\alpha_k}{2\pi}\operatorname{res}_{z_k}f(z)\\
& = \sum_{k\geqslant 1} n_{z_k}(C) \operatorname{res}_{z_k}f(z).
\end{aligned}$$
\subsection*{Acknowledgement}
We would like to thank the referees for their valuable remarks 
which greatly helped to improve this article.

\end{document}